\documentclass[11pt]{article}
\usepackage{amsfonts}
\usepackage{nicefrac}
\usepackage{amsmath,amscd}
\begin{document}

\newtheorem{theorem}{Theorem}[section]
\newtheorem{lemma}[theorem]{Lemma}
\newtheorem{proposition}[theorem]{Proposition}
\newtheorem{corollary}[theorem]{Corollary}
\newtheorem{definition}[theorem]{Definition}
\newtheorem{example}[theorem]{Example}
\newtheorem{conjecture}[theorem]{Conjecture}

\newenvironment{proof}[1][Proof]{\begin{trivlist}
\item[\hskip \labelsep {\bfseries #1}]}{\end{trivlist}}
\newenvironment{remark}[1][Remark]{\begin{trivlist}
\item[\hskip \labelsep {\bfseries #1}]}{\end{trivlist}}

\newcommand{\qed}{\nobreak \ifvmode \relax \else
      \ifdim\lastskip<1.5em \hskip-\lastskip
      \hskip1.5em plus0em minus0.5em \fi \nobreak
      \vrule height0.75em width0.5em depth0.25em\fi}

\title{Unique factorization results for generalized power series }
\author{Noa Lavi \\ \footnotesize{Department of Mathematics and Physics} \\ \footnotesize{ Roma Tre University} \\ \footnotesize {Rome, Italy} \\                 \footnotesize{noa.lavi@mail.huji.ac.il}}

\date{}

\maketitle
\begin{abstract}
Factorization in rings of the form $K((G^{\le 0} )) $ exhibits pathological behavior due to divisibility by monomials. It remains open whether this is the only obstruction and whether the quotient by the ideal generated by all monomials is a unique factorization domain. Prior to this work, the only instances of unique factorization beyond the irreducible elements followed directly from the primality of elements which aren't divisible by any monomial and whose support has order type $\omega$ or $\omega+1$. Using the $RV$ tool from valuation theory as a key ingredient, we establish unique factorization for elements of $ K((\mathbb{R}^{\le 0}))$ which aren't divisible by any monomial and whose support has order type $\omega^{\alpha} $ or $\omega^{\alpha}+1 $ for a large class of non additively principal ordinals $\alpha$.
\end{abstract}
\section{Introduction}
Rings and fields of power series are classical tools in various areas, such as valuation theory. We call \textbf{generalized power series} a formal sum $b = \sum_\gamma b_\gamma t^\gamma$ where the \textbf{exponents} $\gamma$ vary in an ordered abelian group $G$, the \textbf{coefficients} $b_\gamma$ are taken from some field $K$, and its \textbf{support} $\{\gamma \in G : b_\gamma \neq 0\}$ is well-ordered, namely every non-empty subset has a minimum. It is well known that the collection $K((G))$ of such series forms a field, when equipped with the obvious operations of sum and product \cite{Han}. 

In current work we are interested in the ring $K((\mathbb{R}^{\le 0})) $, that is, the general power series with non-positive exponents. Rings of this form appear in different contexts; for instance, $\mathbb{Z} + \mathbb{R}((G^{<0}))$ is always an integer part of the field $\mathbb{R}((G))$ with the merit of being closed under truncations, a fact that plays a crucial part in the work of Ressayre and his collaborators. In \cite{LM2022} it was proved that problems regarding irreducibility and primality in $K((G^{\le 0}))$ for arbitrary $G$ (for example, the case of omnific integers) can be reduced to the case of $K((\mathbb{R}^{\le 0})) $.

$K((\mathbb{R}^{\le 0}))$ is clearly not a unique factorization domain, as any monomial $t^{\gamma}$ doesn't admit a factorization into irreducible elements. In \cite{LM2022} it was proved that any $b \in K(({\mathbb{R}}^{\le 0}) $ has a maximal polynomial (that is, an element of finite support) dividing it. 
 It is an open problem whether this is the only obstruction, that is, whether elements whose maximal divisor lies in $K$ admit a unique factorization. A starting point would be such elements with additively principal order type. 
\begin{definition}
  Let $P_{\alpha} $ denote the set of all series $b \in K(\mathbb{R}^{\le 0})$ such that $ot(b)=\omega^{\alpha}$ and the supremum of the support of $b$ (also denoted by \emph{$sup(b)$}) is $0$. A series $b$ in $P_{\alpha}$ is said to be \textbf{principal}.
\end{definition}
As long conjectured by Conway, Berarducci proved in \cite{Ber} that for every $\alpha$, if $b$ or $b-1$ is in $P_{\omega^{\alpha}}$ then $b$ is irreducible. There, Berarducci also made the following conjecture, to which we are addressing in current work:
\begin{conjecture}
Let $J$ be the ideal generated by all the monomials. The quotient of $K(({\mathbb{R}}^{\le 0}))$ by $J$ is a UFD. 
\end{conjecture}
In \cite{Pit} Pitteloud has proved the following: 
\begin{theorem} \label{Prime} \cite{Pit}
 For every $b \in P_{1}$ we have $b,b+1$ prime. 
\end{theorem}
In \cite{LM2017} it was proved that for every $b \in P_2 \cup P_3$ both $b, b+1$ have a unique factorization into irreducible elements. Apart from this rather direct consequence of the theorem above, no non-trivial instances (that is, non irreducible elements) of unique factorization were previously known.

In current work we generalize the above and prove the following theorem:
\begin{theorem}
Let $\alpha$ be an ordinal of the following forms:
\begin{enumerate}
\item $\alpha = \omega^{\beta_1}+\omega^{\beta_2}$ where $\beta_1 \ge \beta_2 $
\item $\alpha = \omega^{\beta_1}+\omega^{\beta_2}+1$ where $\beta_1 \ge \beta_2$ 
\item $\alpha = \sum_{i=1}^n \omega^{\beta_i}  $ where $\beta_1 > \ldots > \beta_n $
\end{enumerate}
If $b\in P_{\alpha}$ then $b,b+1$ admit a unique factorization into irreducible elements.
\end{theorem}

A main ingredient of the proof is the $RV$ structure of the semi-valuation introduced in \cite{Ber}. We prove that for every $\alpha$ as above unique factorization holds in the $RV$ and then show that, for the ordinals under consideration, this property lifts to $K((\mathbb{R}^{\le 0}))$. We introduce a notion of \emph{Cantor complexity}, defined in terms of the Cantor normal form of the ordinal. This invariant provides the necessary induction parameter: unique factorization in the $RV$-structure for ordinals of a given complexity is reduced to the corresponding statement for strictly lower complexity.

\section{Preliminaries}
We assume familiarity of the reader with the class of ordinals and the classical (non-commutative) operations on them. We recall that an ordinal is \emph{additively principal} if it cannot be represented as a sum of two smaller ordinals. 
\begin{definition}
Let $\alpha$ an ordinal and let $\beta$ the maximal ordinal such that $\alpha \ge \omega^{\beta} $. Then we define $deg(\alpha)=\beta$.
\end{definition}

For every $\alpha$ there is a unique $\gamma$ such that $\alpha = \omega^{deg(\alpha)}+\gamma$. Repeating the above for $\gamma$ allows us to obtain a representation $\alpha = \omega^{\beta_1}+ \ldots + \omega^{\beta_n} $ where $\beta_1 \ge \ldots \ge \beta_n $. This is the \emph{Cantor Normal Form} of $\alpha$. The definition of Cantor Normal Form of an ordinal gives rise to a natural commutative sum and product operations of ordinals. 

Given $\alpha = \omega^{\gamma_1} + \omega^{\gamma_2} + \ldots + \omega^{\gamma_n}$ and $\beta = \omega^{\gamma_{n+1}} + \omega^{\gamma_{n+2}} + \ldots + \omega^{\gamma_{n+m}}$ in Cantor Normal Form, let $\pi$ be a permutation of the integers $1, \ldots, n+m$ such that $\gamma_{\pi(1)} \geq \ldots \geq \gamma_{\pi(n+m)}$. Then the \emph{Hessenberg sum}, denoted by $\alpha \oplus \beta$, is defined to be $\omega^{\gamma_{\pi(1)}}+ \ldots + \omega^{\gamma_{\pi(n)}} $. The \emph{Hessenberg product}, denoted by $\alpha\odot\beta $ is defined to be \[\bigoplus_{1\le i \le m, 1\le j \le n}\omega^{\alpha_i\oplus \beta_j} \] where $\omega^{\alpha_1}+\ldots + \omega^{\alpha_m} $, $\omega^{\beta_1}+\ldots +\omega^{\beta_n} $ are the cantor normal forms of $\alpha, \beta$ compatibly. 

We note that for every $\alpha,\beta$ we have:
\begin{itemize}
\item $\alpha+\beta \le \alpha \oplus \beta$
\item $\alpha\beta \le \alpha \odot \beta$
\item If $\alpha$ is multiplicatively principal, $\beta$ additively principal and $\beta \le \alpha$ then $\alpha \odot \beta = \alpha\beta $

\end{itemize}
\begin{lemma} \label{admul} \cite{Ber}
If $B,C$ are well ordered subsets of some ordered abelian group then the following holds:
\begin{itemize}
\item $ot(B\cup C)\le ot(B)\oplus ot(C) $
\item $ot(B+C) \le ot(B)\odot ot(B) $

\end{itemize}

\end{lemma}

We are now ready to define a semi-valuation which obtains values in $\langle ON\cup \{ -\infty\}, \oplus \rangle $. Let $J$ be the (proper) ideal of $K((\mathbb{R}^{\le 0}))$ generated by the series of the form $t^x$ for $x < 0$. From now on, let $v$ be the following: \\ 
\[ v(b) := \begin{cases}
    -\infty                                      & \text{if } b \in J,                   \\
    0                                      & \text{if } b \in (J + K) \setminus J, \\
    deg(\min\{ot(c) : c \equiv b \mbox { mod } J + K\}) & \text{otherwise}.
  \end{cases} \]
In the case that $b$ is principal, we note that $ot(b)=\omega^{v(b)} $. \\
$v$ is the additive version of the semi-valuation $v_J$ defined in \cite{Ber}. That is, $v(bc)=v(b)\oplus v(c) $. 

We define $J_{\alpha} = \{b \in K((\mathbb{R}^{\le 0})) : v(b) < \alpha \} $. clearly, $J_0 = J $. 

Let $B$ an ordered set. By \emph{almost every} $\gamma \in Supp(b)$ we would mean every $\gamma \in B$ beside a subset of order type strictly smaller than $ot(B)$.  
We recall that $rv(b)=rv(c)$ if and only if $v(b-c) < v(b) $, and $RV$ would be $K((\mathbb{R}^{\le 0}))$ modulo this equivalence relation. For $b,c \in P_{\alpha}$, $rv(b)=rv(c)$ is equivalent of saying that $b_{\gamma}= c_{\gamma}$ for almost every $\gamma \in Supp(b)$. 
 We define \[RV^{\alpha} = \{B \in RV : v(B) = \alpha\} \cup \{rv(0)\}.\] We note that $RV^{\alpha} = J_{\alpha+1} \mbox{ mod } J_{\alpha} $. 
In \cite{Pit}, in order to prove Theorem \ref{Prime}, the author proves first the following:
\begin{proposition} (\cite[Prop 3.2]{Pit}  ) \label{RVPrime}
Let $B \in RV$ such that $v(B)=1 $. Then for every $C,D \in RV$, if $B | CD$ then $B | C$ or $B | D$. 
\end{proposition} 
In this work we also prove first unique factorization in the $RV$ in order to deduce from it unique factorization.

We give now several definitions and propositions that will serve us for the inductive argument in the proof of the main theorem of the article.  
\begin{definition}
Let $v(b) = \omega^{\alpha_1} + \ldots + \omega^{\alpha_n} $ Such that $\alpha_1 \ge \ldots \ge \alpha_n $. 
\begin{itemize}
\item  $n$ is the \emph{cantor complexity of $v$} denoted also by $cc(b)$. 
\item $cf(b) = \{\alpha_1, \ldots , \alpha_n \} $

\end{itemize}

\end{definition}
\begin{definition} \cite{Ber}
Let $v(b)=\sum_{i=1}^n \omega^{\alpha_i}$ where $\alpha_1 \ge \ldots \ge \alpha_n \ge 0 $. 
\begin{itemize}
\item The \emph{principal value of $b$}, denoted by $v^p(b)$, is defined to be $\omega^{\alpha_n} $
\item The \emph{residual value of $b$}, denoted by $v^r(b)$, is defined to be $ \sum_{i=1}^{n-1}\omega^{\alpha_i}$

\end{itemize}
\end{definition}
Clearly, if $v(a)=v^r(b)$ and $ v(b)\neq v^p(b)$ then $cc(b)= cc(a)+1 $. The following will allows us to find a suitable such $a$ for the inductive step, independently of the choice of the representative of $rv(b)$.  
\begin{definition}
  Given
  $b = \sum_{\beta} b_x t^{x} \in K((\mathbb{R}^{\leq 0}))$
  and $\gamma \in \mathbb{R}^{\leq 0}$, we define:
  \begin{itemize}
    \item the \textbf{truncation} of $b$ at $\gamma$ is
          $b_{\vert \gamma} := \sum_{x \leq \gamma } b_x t^{x}$,
    \item the \textbf{translated truncation} of $b$ at $\gamma$ is
          $b^{\vert \gamma} := t^{-\gamma} b_{\vert \gamma}$.
  \end{itemize}
\end{definition}
It turns out that translated truncations behave like a sort of
`generalized coefficients', as they satisfy the following equation.
\begin{proposition}[{\cite[Lemma 7.5(2)]{Ber}}] \label{conv}
  For all $a, b \in K((\mathbb{R}^{\leq 0}))$ and
  $\gamma \in \mathbb{R}^{\leq 0}$ we have:
  \[
    {(ab)}^{\vert\gamma} \equiv \sum_{\delta + \varepsilon = \gamma}
    a^{\vert\delta} b^{\vert\varepsilon} \mbox{ mod } J
  \]
\end{proposition}
And one further more may obtain a kind of a Leibniz rule.
\begin{proposition} [{\cite[Lemma 7.7]{Ber}}] \label{leib}
  Let $b,c \in K((\mathbb{R}^{\le 0}))$ such that \\$v^p(b) \le v^p(c)$. Then for every $\gamma$ sufficiently close to $0$ we have \[(bc)^{|\gamma} = b^{|\gamma}c+c^{|\gamma}b +r \] where $v(r) < v^r(b) \oplus v(c) < v(bc)$.
\end{proposition}
The following Lemma gives a clear candidate for the inductive step:
\begin{lemma} \label{strprin} \cite{FLLM}
  Suppose $v^p(q)>v^p(p)>1 $. Then for every $\gamma$ close enough to $0$ we have $(pq)^{|\gamma} = p^{|\gamma}q \mbox{ mod } J_{v^r(pq)} $.
\end{lemma}
A special case is then $v(p)< v^p(q)$. Then we obtain the following:
\begin{lemma} [{\cite[Lemma 4.7]{FLLM}]}]\label{smallder} 
Let $a,b \in K((\mathbb{R}^{\le 0}))$ such that \\ $v_J(a)<v_J^p(b)$. Then for every $\gamma$ close enough to $0$ we have \\$(ab)^{|\gamma} = a^{|\gamma}b$ mod $J_{deg_J(b)}$.
\end{lemma}

As we are working on the $RV$, our object in the inductive step would be of the form $pq+r $ where $v(r)<v(pq)$. The following provide us with ``enough'' $\gamma$ for which $r^{|\gamma} \in J_{v^r(pq)}$. We note that in the case where $v(pq)$ is a successor each $\gamma$ close enough to $0$ for which $v(p^{|\gamma}q)=v^r(pq) $ would work, but when it is not the case we may have $v(r^{|\gamma})\ge v^r(pq) $. 
\begin{definition} \cite{FLLM}
For every element $b \in K((\mathbb{R}^{\le 0}))$ we define
 \[Res(b) = \{ \gamma < 0 | v(b^{|\gamma}) = v^r(b)\}, 
 Big^{\alpha}(b) = \{ \gamma < 0 | v_J(b^{|\gamma}) \ge {\alpha}\} .\] 
\end{definition}
\begin{lemma} \label{Resot} \cite[Lemma 6.8]{Ber}
For every principal element $b$ we have \\$ot(Res(b))=\omega^{v^p(b)}$ and $sup(Res(b))=0 $.
\end{lemma}
\begin{lemma} \label{bpaf} \cite[Lemma 4.8]{FLLM} 
For every principal $b$ and ordinal $\beta$,
  \[ \omega^{\beta}ot(\{\gamma : v(b^{|\gamma}) \geq \beta\}) \leq \omega ^{v(b)}. \]
\end{lemma}
\begin{corollary} \label{litBig} 
For every principal $b,r $ such that $v(r)<v(b) $ we have that $ot(Big^{v^r(b)}(r))<{v^p(b)} $.
\end{corollary}

\section{Unique Factorization Results}

We recall from \cite{LM2017} that an element $b$ is called \emph{germ-like} if there exists some $\alpha$ and $k \in K$ such that $b+k \in P_{\alpha} $.
\begin{lemma} \label{otrv}
Suppose $b,c$ are germ-like such that $rv(b)=rv(c)$. Then \\ $ot(b-c)<ot(b) $
\end{lemma}
\begin{proof}
If $b-c$ is germ-like then $ot(b-c)=\omega^{v(b-c)}(+1) $ and by definition is strictly smaller than $\omega^{v(b)}(+1)=ot(b) $. Otherwise, let $S = Supp(b-c)-\{0\}$. As $sup(S)<0$ then $ot(S\cap Supp(b)), ot(S\cap Supp(c)) <\omega^{v(b)}$. Then by Lemma \ref{admul} $ot(S) \le ot(S\cap Supp(b)) \oplus ot(S\cap Supp(c)) $. As $\omega^{v(b)} $ is additively principal then $ot(S)<ot(b) $.
\end{proof}
In the following we show how the $rv$ relation is being preserved under residue in the case where the valuation is a successor. 
\begin{lemma} \label{ResRV}
Let $b,c$ germ-like such that $v(b)=v(c) $, $Res(c)= Res(b)$ and $v^p(b)=1$. Then we have $rv(b)=rv(c)$ if and only if  \[ot(\big\{\gamma \in Res(b)  : rv(b^{|\gamma}) \neq rv(c^{|\gamma}) \big\}< \omega .\]
\end{lemma}
\begin{proof}
If $v(b-c)<v(b)$ then $v(b-c) \le v^r(b) $. Hence, there exists some $\gamma_0<0$ such that for every $\gamma \ge \gamma_0$ we have $v((b-c)^{|\gamma})<v^r(b) $. Hence, \[\big\{\gamma \in Res(b)  : rv(b^{|\gamma}) \neq rv(c^{|\gamma}) \big\}\subset Res(b) \cap (-\infty , \gamma_0) < ot(Res(b))=\omega \]

For the other direction, let $\gamma_0 <0$ be such that for every $\gamma \in Res(b)_{>\gamma_0}$ we have $(b-c)^{|\gamma} \in J_{v^r(b)} $. Then for every $\gamma \in(b-c)_{>\gamma_0}$ we have \\$ v((b-c)^{|\gamma})<v^r(b)$. hence, $v^r(b-c)<v^r(b)$. As Hence, $v^r(b-c)\oplus 1 <v(b) $. As $v^p(b)=1 $ this implies that $v(b)>v(b-c)$ as required. $\qed$
\end{proof}
\begin{remark}
We note that we may always assume $Res(b)=Res(c)$ as \\$rv(b)=rv(c)$ implies that there exists some $\gamma_0 <0$ such that \\$Res(b)_{>\gamma_0}=Res(c)_{>\gamma_0} $.
\end{remark}
In the following we show how one may obtain, in some cases, unique factorization of a germ-like element from unique factorization of its image in the $RV$. 

\begin{proposition} \label{fromrv}
Let $p_1, p_2, q_1, q_2 $ germ-like. Suppose $p_1p_2 = q_1q_2 $ where $cf(q_1)\cap cf(q_2)=\emptyset$, $rv(p_1)= krv(q_1)$, $rv(p_2) = k^{-1}rv(q_2) $ for some $k \in K^{\times} $. Then $p_1=kq_1 $, $p_2 = k^{-1}q_2 $.
\end{proposition}
\begin{proof}
We have $p_1 = k(q_1+r_1) $, $p_2 = k^{-1}(q_2+r_2) $ where $v(r_1) <v(p_1)$, $v(r_2) <v(p_2)$. Hence $(*)r_2q_1 + r_1q_2 +r_1r_2 = 0$.  Moreover, as $p_1,p_2$ are germ-like, we have by Lemma \ref{otrv} $ot(r_1)<v(p_1)$, $ot(v_2)<v(p_2) $.

By the valuation inequality we get $v(r_2q_1)=v(r_1q_2) $. Without loss of generality suppose $v(q_1) < v(q_2) $. If $r_1 \notin J$ then $v(q_2) \le v(q_1) + v(r_2) $. As $cf(q_1) \cap cf(q_2) = \emptyset $ we have that $cf(q_2) \subseteq cf(r_2) $ which contradicts the fact that $v(r_2) < v(q_2) $. Hence $r_1, r_2 \in J $. Let $\gamma = max(sup(r_1), sup(r_2)) $. By $(*)$ we have $(q_1r_2)^{|\gamma} + (q_2r_1)^{|\gamma} \in J$. By Proposition \ref{conv} we obtain \\$q_1r_2^{|\gamma}+q_2r_1^{|\gamma} \in J $. Without loss of generality, suppose $r_1^{|\gamma} \notin J $. As before, we have $cf(q_2) \subset cf(r_2^{|\gamma}) $, which contradicts the fact that $ot(r_2)<v(q_2)$. Hence, $r_1^{|\gamma}, \in J $. By the definition of $\gamma$ we must have $r_1 = 0 $. Then by $(*)$ we have $q_1r_2 = 0$ which implies that also $r_2 = 0$. $\qed$
\end{proof}
\begin{lemma} \label{fromRVPrin}
Let $p$ germ-like such that $v(p)$ is additively principal and let $r,r_1, r_2 \in J_{v(p)}$ germ-like. If $kp(p+r) = (p+r_1)(kp+kr+r_2)$ then \\$r_1=r_2=0$.
\end{lemma}
\begin{proof}
By assumption we have $(*)p(kr_1+r_2)+krr_1+r_1r_2=0 $. Hence, $v(p(kr_1+r_2))=v(krr_1+r_1r_2) $. As $v(p)$ is additively principal we have $v(r_1)+v(r_2), v(r_1)+v(r)<v(p) $. Hence, $r_1, r_2 \in J$. 

\underline{ Case $r \neq 0$}: 

Let $\gamma = max ( sup(r_1), sup(r_2) )$. By $(*)$ we have that  \[(p(kr_1+r_2))^{|\gamma}+(krr_1)^{|\gamma}+(r_1r_2)^{|\gamma}=0 .\] By Proposition \ref{conv} we obtain $p(kr_1^{|\gamma}+r_2^{|\gamma}) + kr{r_1}^{|\gamma} \in J$. Hence, \\$v(p(kr_1^{|\gamma}+r_2^{|\gamma}))=v(kr{r_1}^{|\gamma}) $. As by Lemma \ref{otrv} $ot(r_1), ot(r_2) <ot(p)$ we have that $v(r_1^{|\gamma}), v(r_2^{|\gamma}) <v(p) $. Hence, as $v(p)$ is additively principal, we have $v(kr{r_1}^{|\gamma})<v(p)$. Hence $r_1^{|\gamma},r_2^{|\gamma} \in J $. Hence, $r_1=r_2=0$. 

\underline {Case $r=0$ }:

Let $\gamma = sup(r_1r_2) $. By (*) we have $ (p(kr_1+r_2))^{|\gamma}+(r_1r_2)^{|\gamma}=0$. Hence, $v((p(kr_1+r_2))^{|\gamma})=v((r_1r_2)^{|\gamma})$. As $v(p)$ is additively principal then for every $\delta > \gamma$ we have $v(p^{|\gamma - \delta})(kr_1+r_2)^{|\delta} <v(p) $. Hence, by Proposition \ref{conv} we have $v((p(kr_1+r_2))^{|\gamma})= v(p (kr_1+r_2)^{|\gamma}) $. Hence, \\$ v(p (kr_1+r_2)^{|\gamma})=v((r_1r_2)^{|\gamma})$. By Lemma \ref{admul} we have \\$ot(r_1r_2)\le ot(r_1)\odot ot(r_2) < ot(p) $. Hence $v((r_1r_2)^{|\gamma})<v(p)$. Hence, $(r_1r_2)^{|\gamma} \in J $. Hence, $r_1r_2=0$. Hence, without loss of generality $r_1=0$. By (*) we have $pr_2=0 $. Hence, $r_2=0$. 
 $\qed$
\end{proof}
\begin{lemma} \label{sucRV}
Let $p,b,r$ germ-like such that $v(p)=\omega^{\alpha}$, $v(b)=1$,\\$v(r) \le v(p)$. Suppose $r_1 \in J_{\omega^{\alpha}}$, $r,r_2 \in J_{\omega^{\alpha}+1}$ such that \\ $p(bp+r)=(p+r_1)(bp+r+r_2) $. Then $r_1 = r_2 = 0$.
\end{lemma}
\begin{proof}
By assumption we have $(*)p(br_1+r_2)+rr_1+r_1r_2=0 $. As \\$v(rr_1)<v(pbr_1)$ and $v(r_1r_2)<v(pr_2)$ we have $rv(br_1)=rv(-r_2) $, hence $rv(b) | rv(r_2) $. Hence, $r_2 \in J_{\omega^{\alpha}} $. If $r \in J_{\omega^{\alpha}}$ we conclude as in the proof of Lemma \ref{fromRVPrin}. Otherwise, Let $n\ge 1$ be maximal such that $rv(b)^n | rv(r_2) $ and let $r_2 = b^nr'_2+c$ for some $c$ such that $v(c)<v(r_2)$. By Proposition \ref{RVPrime} we have that $rv(b)^{n-1}|rv(r_1)$ and $r_1 = b^{n-1}r'_1+d $ where $rv(b)$ does not divide $rv(r'_1) $, $v(d)<v(r_1)$. Clearly, \[rv(p(b^{n}r'_1+b^nr'_2) +pbd + pc)=rv(p(b^{n}r'_1+b^{n}r'_2)) .\] 
As $v(r)>v(r_2)$ we have 
 \[rv(p)rv(b)^{n}rv(r'_1+r'_2)= -rv(r)rv(b^{n-1})rv(r'_1) .\] Hence, $rv(p)rv(b)rv(r'_1+r'_2)=-rv(r)rv(r'_1) $. As $v(r)=v(p)$ additively principal we have that $rv(b) | rv(r'_1)$, a contradiction. $\qed$
\end{proof}

\subsection{The case $\omega^{\alpha}+\omega^{\beta}$}
In this subsection we prove unique factorization for germ-like elements with valuation $\omega^{\alpha}+\omega^{\beta} $, by proving unique factorization in the $RV$ sort. The following Proposition is fundamental. 
\begin{proposition} \label{2prin}
Suppose $rv(c_1b_1)=rv(c_2b_2) $ where $v^p(b_1) = v^p(b_2) $ and $0<v(c_1), v(c_2) < v^p(b_1) $. Then $rv(c_1)=krv(c_2) $, $rv(b_1)=k^{-1}rv(b_2) $ for some $k \in K^{\times}$.
\end{proposition}
\begin{proof}
We prove by induction on $n=cc(c_1) $.

In case $n=0$ we have $c_1, c_2 \in J+K $. Choose $k = \frac{sup(c_1)}{sup(c_2)} $. 

Otherwise, let $c_1b_1 = c_2b_2 + r $ where $v(r) < v(c_1b_1) $. By Corollary \ref{litBig} we have $\gamma \in Res(c_1b_1) -  Big^{v^r(c_1b_1)}(r) $ arbitrarily close to $0$. By Lemma \ref{strprin} we have \[(c_1b_1)^{|\gamma} = c_1^{|\gamma}b_1 \mbox{ mod } J_{v^r(c_1b_1)} , v(r^{|\gamma}) < v^r(c_1b_1) , (c_2b_2)^{|\gamma} = c_2^{|\gamma}b_2 \mbox { mod } J_{v^r(c_1b_1)} .\] Hence, $rv(c_1^{|\gamma}b_1) = rv(c_2^{|\gamma}b_2 )$. As $\gamma \in Res(c_1) $ we obtain by the induction assumption that $rv(b_1) = k^{-1}rv(b_2) $ for some $k \in K^{\times} $. Hence also\\$rv(c_1) = krv(c_2) $. 
\end{proof}
\begin{remark}
In this case, one could prove also by induction just on $v(c_1)$ rather than $cc(c_1)$. 
\end{remark}
\begin{corollary}
Suppose $c_1b_1=c_2b_2$ where $v^p(b_1) = v^p(b_2) $ and\\$0<v(c_1), v(c_2) < v^p(b_1) $. Then $c_1=kc_2 $, $b_1=k^{-1}b_2 $ for some $k \in K^{\times}$.
\end{corollary}
\begin{proof}
As $v(c_1)<v^p(b_1)$ we have $cf(c_1) \cap cf(b_1)=\emptyset  $. By Proposition \ref{fromrv} we obtain the required. 
\end{proof}
\begin{corollary}
Let $b$ germ-like such that $v(b) = \omega^{\alpha} +1 $ or $v(b) = \omega^{\alpha}+2 $ where $\alpha > 0 $. Then $b$ admits a unique factorization into irreducible elements.
\end{corollary}
\begin{proof}
Suppose $b=pq $ where $p,q$ are irreducible, $v(p) = \omega^{\alpha}+1$, \\$v(q)=1$. Then, by Theorem \ref{Prime} $q$ is prime. Hence, if $b=p'q'$ where $p',q'$ are irreducible then without loss of generality we have $q = kq' $ for some $k \in K^{\times}$. Hence, $p = k^{-1}p' $, as required. In all other cases we have $v(q)<v^p(p)$.  $\qed$ 
\end{proof}
We are now left to prove unique factorization for the case $\alpha = \beta >1 $. We start by proving a crucial Lemma. 
\begin{lemma} \label{span}
Let $b,c \in P_{\omega^{\alpha}} $. If there exists some $\gamma \in Res(bc)$ such that $\gamma \notin Res(b)$ then $rv(b) \in Sp_K(rv(c)) $.
\end{lemma} 
\begin{proof}
Suppose that there exists such $\gamma$. Then by Proposition \ref{leib} we have that $v(b^{|\gamma}c + bc^{|\gamma})=\omega^{\alpha} $ where $v(b^{|\gamma}c), v(bc^{|\gamma}) > \omega^{\alpha} $. Hence, \\$rv(b^{|\gamma}c)=rv(-bc^{|\gamma}) $. As $v^p(b)=v^p(c)=v(b)=v(c) $ we have that $v(b^{|\gamma}), v(c^{|\gamma})<v^p(b) $. Hence, by Proposition \ref{2prin} we have some $k \in K^{\times}$ such that $rv(b)= krv(c) $, as required. 
\end{proof}
We are now ready to prove unique factorization in th $RV$ sort. 
\begin{proposition}
Suppose $rv(a)rv(b)=rv(c)rv(d)$ where $a,b,c,d \in P_{\omega^{\alpha}}$. Then there exists some $k \in K^{\times}$ such that without loss of generality \\$rv(b)=krv(d) $.
\end{proposition}
\begin{proof}
If $rv(a) \in Sp_K(rv(b)), rv(c) \in Sp_K(rv(d))$ then we have \\ $k_1rv(b)^2 = k_2rv(d)^2$ for some $k_1, k_2 \in K^{\times}$. Hence, without loss of generality, $rv(d)=\sqrt{\frac{k_1}{k_2}}rv(b)$. For almost all $\gamma \in Supp(b)\cap Supp(d)$ we have $\frac{d_{\gamma}}{b_{\gamma}} = \sqrt{\frac{k_1}{k_2}}$. Hence, $k= \sqrt{\frac{k_1}{k_2}} \in K^{\times} $.

Suppose without loss of generality $rv(a) \in Sp_K(rv(b)), rv(c) \notin Sp_K(rv(d)) $. Let $rv(b)=krv(a) $. By Lemma \ref{span} for every $\gamma \in Res(cd) $ we have \\$c^{|\gamma} = c_{\gamma}, d^{|\gamma}=d_{\gamma} $. By Proposition \ref{leib} for almost all $\gamma \in Supp(b)$ we obtain $2kb_{\gamma}rv(b) = c_{\gamma}rv(d)+ d_{\gamma}rv(c)$. Hence, \[rv(b) = \frac{c_{\gamma}}{2kb_{\gamma}}rv(d) + \frac{d_{\gamma}}{2kb_{\gamma}}rv(c) .\] As $\{rv(c), rv(d)\} $ is a $K$-linearly independent set then for almost every \\$\gamma_1 \neq \gamma_2 \in Supp(b) $ we have $\frac{d_{\gamma_1}}{2kb_{\gamma_1}} = \frac{d_{\gamma_2}}{2kb_{\gamma_2}}$ . Denote this element by $k'$. Hence for almost all $\gamma \in Supp(b)$ we have $d_{\gamma}=k'2kb_{\gamma} $, as required.

The remaining case is both $\{rv(a), rv(b)\},\{rv(c),rv(d)\} $ are $K$-linearly  independent sets. By Proposition \ref{leib} we obtain for almost every $\gamma \in Supp(b)$ we have $a_{\gamma}rv(b) + b_{\gamma}rv(a) = c_{\gamma}rv(d) + d_{\gamma}rv(c)  $. If without loss of generality $\{rv(a), rv(b), rv(d)\} $ is also a $K$-linearly independent set we have \[rv(c) = \frac{a_{\gamma}}{d_{\gamma}}rv(b) + \frac{b_{\gamma}}{d_{\gamma}}rv(a) + \frac{-c_{\gamma}}{d_{\gamma}}rv(d).\] Denote $k' = \frac{b_{\gamma}}{d_{\gamma}}$. Then we have $d_{\gamma}=k'^{-1}b_{\gamma} $ for almost every $\gamma \in Supp(b)$, as required. 

If $rv(c), rv(d) \in Sp_K(\{rv(a), rv(b)\}) $ then let $rv(c) = \lambda_1rv(a)+\lambda_2rv(b)$, $rv(d) = \lambda_3rv(a) + \lambda_4rv(b) $. Then for by Proposition \ref{leib} for almost every $\gamma \in Supp(b)$ we have \begin{multline*}a_{\gamma}rv(b) +b_{\gamma}rv(a)=(\lambda_3a_{\gamma}+\lambda_4b_{\gamma})(\lambda_1rv(a) + \lambda_2rv(b))+\\+(\lambda_1a_{\gamma}+\lambda_2b_{\gamma})(\lambda_3rv(a) + \lambda_4 rv(b)) \end{multline*}  where $ \lambda_1, \lambda_2, \lambda_3, \lambda_4 \in K^{\times}$. Hence, \[(2\lambda_3\lambda_1a_{\gamma}+(\lambda_4\lambda_1+\lambda_2\lambda_3 - 1)b_{\gamma} )rv(a) + (2\lambda_4\lambda_2b_{\gamma }+(\lambda_3\lambda_2+\lambda_1\lambda_2 -1)a_{\gamma} )rv(b) = 0 .\]   
As $\{rv(a),rv(b) \}$ is a $K$-linearly independent set then we have \[2\lambda_3\lambda_1a_{\gamma}+(\lambda_4\lambda_1+\lambda_2\lambda_3 - 1)b_{\gamma}=0.\] Hence \[a_{\gamma} = -\frac{\lambda_4\lambda_1+\lambda_2\lambda_3 - 1}{2\lambda_3\lambda_1}b_{\gamma} \] for almost every $\gamma \in Supp(b)$, which contradicts the independence of $\{rv(a), rv(b)\} $. Hence, without loss of generality, $rv(d) \in Sp_K(rv(b))$. $\qed$
\end{proof}
\begin{corollary}
Let $p_1,q_1,p_2,q_2$ germ-like of valuation $\omega^{\alpha}$. Suppose \\$p_1p_2=q_1q_2$ and $rv(p_1)=krv(q_1), rv(p_2)=k^{-1}rv(q_2)$ for some $k \in K^{\times} $. Then $p_1=kq_1, p_2=k^{-1}q_2$. 
\end{corollary}
\begin{proof}
Let $r_1, r_2 \in J_{\omega^{\alpha}}$ such that $p_1 = k(q_1 +r_1)$, $q_2 = k^{-1}(p_2+r_2)$. Hence $r_2q_1+r_1q_2+r_1r_2=0 $. Hence $rv(r_2q_1)=-rv(r_1q_2)$. Hence, by Proposition \ref{2prin} we have that $rv(r_2)=k'rv(-r_1), rv(q_1)=k'^{-1}rv(q_2) $. By Lemma \ref{fromRVPrin} we have that $r_1=r_2=0$. 
\end{proof}

\subsection {The strictly decreasing case}
We prove that if $b$ is germ-like and $v(b)=\sum_{i=1}^{cc(b)}\omega^{\alpha_i}$ where $\alpha_1 > \ldots > \alpha_{cc(b)} $ then $b$ has a unique factorization into irreducible elements. 
As before, we prove first unique factorization in the $RV$. 
\begin{theorem} \label{strictdec}
Let $p_1, \ldots , p_m, q_1, \ldots, q_n$ germ-like of valuation strictly greater than $0$ be such that  $rv(p_1), \ldots , rv(p_m), rv(q_1) , \ldots , rv(q_n)$ are irreducible. Suppose $p_1\ldots p_m = b = q_1\ldots q_n +r $ where $cc(b)=|cf(b)| $ , $v(r) < v(b)$, and without loss of generality $v^p(b)=v^p(p_1)=v^p(q_1) $. Then $n=m$ and there exists some permutation $\pi$ on $\{1, \ldots , m\} $ such that for every $1 \le i \le m$ we have $rv(p_i) = k_i rv(q_{\pi(i)}) $ where $k_i \in K^{\times} $. 
\end{theorem}
\begin{proof}
We prove by induction in $cc(b)$. \\
If $cc(b)=1 $ then $b$ is irreducible and clearly $rv(b+r)=rv(b)$.

For $cc(b)>1$ suppose that there exist some $i,j$ such that \\$rv(q_i)=krv(p_j) $ for some $k \in K^{\times}$. Then we have \[rv(kq_1 \ldots  q_{i-1}q_{i+1}\ldots q_n)=rv(p_1\ldots p_{j-1}p_{j+1}\ldots p_m) .\] Hence, by the induction assumption we have the required.

Otherwise, let $\Gamma = Res(b) - Big^{v^r(b)}(r) $. Let $\{p_{\gamma}\}_{\gamma\in Res(b)},\{ q_{\gamma}\}_{\gamma \in Res(b)} $ germ-like such that \[p_1 = \sum_{\gamma \in Res(b)}p_{\gamma}t^{\gamma}, q_1 = \sum_{\gamma \in Res(b)}q_{\gamma}t^{\gamma} .\] By Lemma \ref{admul} \[ot(\sum_{\gamma \in Res(b)-\Gamma}p_{\gamma}t^{\gamma})\le ot(Big^{v^r(b)}(r))\odot v^r(b),\] \[ot(\sum_{\gamma \in Res(b)-\Gamma}q_{\gamma}t^{\gamma})\le ot(Big^{v^r(b)}(r))\odot v^r(b) \] Since by corollary \ref{litBig} we have that $ ot(Big^{v^r(b)}(r))<\omega^{v^p(b)}$  then \[(\sum_{\gamma \in Res(b)-\Gamma}p_{\gamma}t^{\gamma})p_2\ldots p_m, (\sum_{\gamma \in  Res(b)-\Gamma}q_{\gamma}t^{\gamma})q_2\ldots p_n \in J_{v(b)}. \]

For every $\gamma \in \Gamma$ we have by Lemma \ref{strprin} that \[ p_{\gamma}p_2 \ldots p_m = q_{\gamma} q_2\ldots q_n +r_{\gamma}\] where $v(r_{\gamma})<v^r(b) $. By the induction assumption we obtain that either there exist some $1<i,j$ such that $rv(q_i)=krv(p_j) $ for some $k \in K^{\times}$ or $cf(p_2\ldots p_m)\subset cf(q_{\gamma}), cf(q_2\ldots q_n)\subset cf(p_{\gamma}) $. In the first case we are done, in the latter we have that either $v^p(b^{|\gamma})=v^p(p_{\gamma}) $ or $v^p(b^{|\gamma})=v^p(q_{\gamma}) $. Assume without loss of generality the latter. We have $p_{\gamma} = a_{\gamma}a_2\ldots a_n $ where $a_2,\ldots , a_n $ are germ-like with $rv(a_i)=k_irv(q_i) $ where $k_i \in K^{\times}$ for every $2\le i \le n$, and $a_{\gamma}$ is either $1$ or germ-like. For every $2 \le i \le n$ we have by Lemma \ref{otrv} that $ot(a_i -k_iq_i)<ot(q_i). $ Since $v^p(q_i)>v^p(b^{|\gamma})$  by Lemma \ref{admul} we have that \[ot(a_{\gamma}p_2\ldots p_m k_2q_2\ldots k_{i-1}q_{i-1}a_{i+1}\ldots a_n(a_i-k_iq_i)) <\omega^{v^r(b^{|\gamma})}. \] Again by Lemma \ref{admul} and the fact that $v^r(b^{|\gamma}) $ is additively principal we have that  \[ot(p_{\gamma}p_2\ldots p_m - q_2\ldots q_nk_{\gamma}a_{\gamma}p_2\ldots p_m)<\omega^{v^r(b^{|\gamma})} \] where $k_{\gamma}=k_1\ldots k_n$. Hence, \[(\sum_{\gamma \in \Gamma}p_{\gamma}t^{\gamma}) p_2\ldots p_m = (\sum_{\gamma \in \Gamma}k_{\gamma}a_{\gamma}t^{\gamma})p_2\ldots p_mq_2\ldots q_n+\sum_{\gamma \in \Gamma}r'_{\gamma}t^{\gamma} \]
Where $ot(r'_{\gamma})<\omega^{v^r(b^{|\gamma})} $ for every $\gamma \in \Gamma$. Since $ot(\Gamma)=\omega^{v^p(b)}<\omega^{v^p(b^{|\gamma})} $ we have by Lemma \ref{admul} for some $\gamma_0 \in \Gamma$ that \[ot(\sum_{\gamma \in \Gamma}r'_{\gamma}t^{\gamma})\le \omega^{v^r(b^{|\gamma_0})}\odot\omega^{v^p(b)} <\omega^{v^r(b)} .\] Hence, \[(\sum_{\gamma \in \Gamma}q_{\gamma}t^{\gamma})q_2\ldots q_n = (\sum_{\gamma \in \Gamma}k_{\gamma}a_{\gamma}t^{\gamma})p_2\ldots p_mq_2\ldots q_n + r' \]
where $r' \in J_{v(b)}$. Hence, for some $k \in K^{\times}$ we have \[rv(q_1)=krv(\sum_{\gamma \in \Gamma}k_{\gamma}a_{\gamma}t^{\gamma})rv(p_2)\ldots rv(p_m)\]
Hence, $rv(q_1)$ is not irreducible, which is a contradiction. $\qed$
\end{proof}
\begin{corollary}
Let $b$ germ-like such that $|cf(b)|=cc(b)$. Then $b$ has a unique factorization into irreducible elements.
\end{corollary}
\begin{proof}
Suppose $p_1, p_2, q_1, q_2$ are germ-like such that $p_1p_2=b=q_1q_2 $. By Theorem \ref{strictdec} we know that $rv(b)$ has a unique factorization. Hence, we may assume $rv(p_1)=krv(q_1), rv(p_2)=k^{-1}rv(q_2) $ for some $k \in K^{\times}$. As $|cf(b)|=cc(b) $ then clearly $cf(p_1)\cap cf(p_2) = \emptyset $. Hence, by Proposition \ref{fromrv}, we have $p_1=kq_1, p_2= k^{-1}q_1 $. 
\end{proof}
The proof of Theorem \ref{strictdec} may be repeated in a simpler way for the case of $\alpha = 2\omega^{\beta}+1$ for $\beta > 0$.  
\begin{proposition} \label{limsuc}
Let $\alpha$ be as above. Then for every $b \in P_{\alpha}$ we have that $rv(b)$ has a unique factorization into irreducible elements.
\end{proposition}
\begin{proof}
$p_1, p_2,q_1,q_1$ germ-like such that  $rv(p_1), rv(p_2), rv(q_1) ,rv(q_2)$ are irreducible. Suppose $p_1p_2 = b = q_1q_2 +r $ where $v(r) < v(b)$, and without loss of generality $v^p(p_1)=v^p(q_1)=1 $. Let $\{p_{\gamma}\}_{\gamma\in res(b)},\{ q_{\gamma}\}_{\gamma \in Res(b)} $ germ-like such that \[p_1 = \sum_{\gamma \in Res(b)}p_{\gamma}t^{\gamma}, q_1 = \sum_{\gamma \in Res(b)}q_{\gamma}t^{\gamma} .\]
For every $\gamma \in Res(b)$ we have by Lemma \ref{strprin} that $p_{\gamma}p_2  = q_{\gamma} q_2 +r_{\gamma}$ where $v(r_{\gamma})<v^r(b) $.
If there exists some $\gamma \in Res(b)$ for which $rv(p_{\gamma})=krv(q_{\gamma})$ for some $k \in K^{\times}$ then $rv(p_2)=k^{-1}rv(q_2)$. Hence $rv(p_1)=krv(q_1)$ and we are done. Otherwise, by unique factorization for every $\gamma \in Res(b)=Res(q_1)$ we have $rv(p_{\gamma}) = k_{\gamma}rv(q_2)$ for some $k_{\gamma} \in K^{\times}$. Let \[p = (\sum_{\gamma \in Res(b)}k_{\gamma}t^{\gamma})q_2.\]
As $Res(p)=Res(b)=Res(p_1)$ and $v^p(p)=v^p(b)=v^p(p_1)=1$ then by Lemma \ref{ResRV} we have $rv(p_1)=rv(p)$. Hence, \[rv(p_1) = rv(\sum_{\gamma \in Res(b)}k_{\gamma}t^{\gamma})rv(q_2).\] Hence, $rv(p_1)$ is not irreducible, a contradiction. $\qed$
\end{proof}
We show now that also in this case we may deduce unique factorization from the unique factorization in the $RV$. 
\begin{lemma}
Let $p_1, q_1$ germ-like of valuation ${\alpha}$ where $\alpha$ is additively principal and $p_2,q_2$ germ-like of valuation ${\alpha}+1$. Suppose $p_1p_2 = q_1q_2$ and $rv(p_1)=krv(q_1)$, $rv(p_2) = k^{-1}rv(q_2)$ for some $k \in K^{\times}$. Then $p_1 = kq_1$, $p_2 = k^{-1}q_2$.
\end{lemma}
\begin{proof}
Suppose $p_1 = k(q_1+r_1)$, $p_2 = k^{-1} (q_2+r_2 )$ where $r_1 \in J_{\alpha}$, \\$r_2 \in J_{\alpha+1}$. By assumption we have $(*)r_1q_2 + r_2q_1 +r_1r_2 =0 $. Hence $v(r_1q_2) = v(r_2q_1) $. Hence, $v(r_2)= v(r_1)+1$, which implies that $r_2 \in J_{\alpha}$. Hence $r_1q_2 + r_2q_1 \in J_{\alpha} $. For every $\gamma \in Res(q_2)$ we have by Proposition \ref{leib} \[(r_1q_2)^{|\gamma} = r_1^{|\gamma}q_2 + r_1q_2^{|\gamma} \mbox { mod } J_{\alpha \oplus v(r_1)}.\] Let $n$ be the minimal such that $v(r_1^{|n\gamma})\oplus 1 <v(r_1^{|(n-1)\gamma}) $. For every $m < n $ we have by Proposition \ref{leib} \[(r_1^{|m\gamma}q_2)^{|\gamma} = r_1^{|(m+1)\gamma}q_2 + r_1^{|m\gamma}q_2^{|\gamma} \mbox{ mod } J_{\alpha \oplus v(r_1^{|m\gamma})} \] and by Lemma \ref{smallder} we have \[(mr_1^{|(m-1)|\gamma}q_2^{|\gamma})^{|\gamma} = mr_1^{|m\gamma}q_2^{|\gamma} \mbox { mod } J_{\alpha}. \] Hence, we have \[ (r_1q_2)^{|(m+1)\gamma} = r_1^{|(m+1)\gamma}q_2 + (m+1)r_1^{|m\gamma}q_2^{|\gamma} \mbox { mod } J_{\alpha \oplus v(r_1^{|m\gamma})} \]
Hence, 
\[rv((r_1q_2)^{|n\gamma})=rv(nr_1^{|(n-1)\gamma}q_2^{|\gamma} ) \] and $v((r_1q_2)^{|n\gamma})\ge \alpha $. Hence, $rv((r_1q_2)^{|n\gamma})=rv((r_2q_1)^{|n\gamma}) $. By Lemma \ref{smallder} we have \[(r_2q_1)^{|n\gamma} = r_2^{|n\gamma}q_1 \mbox{ mod } J_{\alpha}  .\] Hence for every $\gamma \in Res(q_2)$ we have \[rv(nr_1^{|(n-1)\gamma}q_2^{|\gamma}) = rv(-r_2^{|n\gamma}q_1).\] Hence by Proposition \ref{2prin} we have $rv(q_2^{|\gamma})=k_{\gamma}rv(q_1)$ for some $k_{\gamma} \in K^{\times}$. Let \[b = \sum_{\gamma \in Res(q_2)}k_{\gamma}t^{\gamma}.\] As $Res(q_1b)=Res(q_2)$ and $v^p(q_2)=1 $ then by Lemma \ref{ResRV} we obtain $rv(q_2) = rv(q_1)rv(b)$. As $v(b)=1$ then by Lemma \ref{sucRV} we are done. $\qed$
\end{proof}

\end{document}